\documentclass[a4paper,11pt,centertags,reqno,psamsfonts,draft]{amsart}

\usepackage{letltxmacro}        % http://ctan.org/pkg/letltxmacro
\usepackage{ifthen}		        % http://ctan.org/pkg/ifthen
\usepackage[headings]{fullpage} % http://ctan.org/pkg/fullpage
\usepackage{setspace}		    % http://ctan.org/pkg/setspace
\usepackage[inline]{enumitem}   % http://ctan.org/pkg/enumitem
\usepackage{array}				% http://ctan.org/pkg/array
\usepackage{fancybox}		    % http://ctan.org/pkg/fancybox     
\usepackage{graphicx}           % http://ctan.org/pkg/graphicx
\usepackage{subfigure}		    % http://ctan.org/pkg/subfigure
\usepackage{amsmath}	        % http://ctan.org/pkg/amsmath
\usepackage{amsthm}				% http://ctan.org/pkg/amsthm
\usepackage{amssymb}				
\usepackage{mathtools}			% http://ctan.org/pkg/mathtools
\usepackage{mathrsfs}			% http://ctan.org/pkg/mathrsfs
\usepackage{stmaryrd}			% http://ctan.org/pkg/stmaryrd
\usepackage{yhmath}				% http://ctan.org/pkg/yhmath
\usepackage{accents}			% http://ctan.org/pkg/accents
\usepackage{extarrows}          % http://ctan.org/pkg/extarrows
\usepackage{xfrac}              % http://ctan.org/pkg/xfrac
\usepackage[all]{xy}			% http://ctan.org/pkg/xypic
\usepackage[round,comma,authoryear,sort&compress]{natbib} % http://ctan.org/pkg/natbib

\newcommand{\N}{\mathbb{N}}
\newcommand{\Z}{\mathbb{Z}}

\newcommand{\R}{\mathbb{R}}

\newcommand{\C}[2]{
\ifthenelse{#1=0 \and #2=0}{\textsf{\upshape C}}
{\ifthenelse{#2=0}{\textsf{\upshape C}^{#1}}
{\textsf{\upshape C}^{#1,#2}}}
}

\DeclarePairedDelimiter{\floor}{\lfloor}{\rfloor}

\renewcommand{\d}{\mathrm{d}}

\newcommand{\Laplace}[1]{\ensuremath{\mathscr{L}_{#1}}}

\newcommand{\e}{\mathrm{e}}

\newcommand{\E}{\textsf{\upshape E}}

\renewcommand{\P}{\textsf{\upshape P}}

\newcommand{\indicator}[1]{\mathbf{1}_{#1}}

\DeclarePairedDelimiter{\PredQVar}{\langle}{\rangle}

\newcommand{\filt}[1]{\mathfrak{#1}}
\newcommand{\sigalg}[1]{\mathscr{#1}}

\newcommand{\Sf}{\ensuremath{\mathscr{S}_{\textsf{\upshape f}}}}

\newcommand{\M}{\ensuremath{\mathscr{M}}}
\newcommand{\Mloc}{\ensuremath{\mathscr{M}_{\textsf{\upshape loc}}}}

\newcommand{\Hp}[1]{\ensuremath{\mathscr{H}^#1}}

\newcommand{\scale}{\ensuremath{\mathfrak{s}}}
\newcommand{\speed}{\ensuremath{\mathfrak{m}}}

\makeatletter
\let\oldr@@t\r@@t
\def\r@@t#1#2{%
\setbox0=\hbox{$\oldr@@t#1{#2\,}$}\dimen0=\ht0
\advance\dimen0-0.2\ht0
\setbox2=\hbox{\vrule height\ht0 depth -\dimen0}%
{\box0\lower0.4pt\box2}}
\LetLtxMacro{\oldsqrt}{\sqrt}
\renewcommand*{\sqrt}[2][\ ]{\oldsqrt[#1]{#2}}
\makeatother

\theoremstyle{plain}
\newtheorem{theorem}{Theorem}
\newtheorem{lemma}[theorem]{Lemma}

\newtheorem{corollary}[theorem]{Corollary}
\theoremstyle{definition}

\newtheorem{example}[theorem]{Example}

\newtheorem{problem}[theorem]{Problem}
\theoremstyle{remark}

\numberwithin{theorem}{section}
\numberwithin{equation}{section}
\numberwithin{figure}{section}
\numberwithin{table}{section}

\graphicspath{{Figures/}}

\bibpunct{(}{)}{,}{a}{}{,}

\allowdisplaybreaks[4]

\begin{document}
\title[Weak Tail Conditions for Local Martingales]{Weak Tail Conditions for Local Martingales}

\author{Hardy Hulley and Johannes Ruf}

\address{Hardy Hulley\\
Finance Discipline Group\\
University of Technology Sydney}

\email{hardy.hulley@uts.edu.au}

\address{Johannes Ruf\\
Department of Mathematics\\
University College London}

\email{j.ruf@ucl.ac.uk}

\thanks{This paper is dedicated to the memory of Nicola Bruti-Liberati. The authors thank Sam Cohen, Ioannis Karatzas, Kostas Kardaras, R\"udiger Kiesel, Alex Novikov, and Eckhard Platen for several helpful discussions. Johannes Ruf gratefully acknowledges the support provided by the Bruti-Liberati Visiting Fellowship Fund, and thanks the Finance Discipline Group at the University of Technology Sydney for its hospitality. He also acknowledges the generous support provided by the Oxford-Man Institute of Quantitative Finance at the University of Oxford.}

\subjclass[2010]{Primary: 60G44}

\keywords{Local martingales; uniformly integrable martingales; weak tail of the supremum}

\date{\today}

\begin{abstract}
The following conditions are necessary and sufficient for an arbitrary c\`adl\`ag local martingale to be a uniformly integrable martingale:
\begin{enumerate*}[label=(\Alph*)]
\item
The weak tail of the supremum of its modulus is zero;
\item
its jumps at the first-exit times from compact intervals converge to zero in $L^1$, on the events that those times are finite; and
\item
its almost sure limit is an integrable random variable.
\end{enumerate*}
\end{abstract}

\maketitle

\section{Introduction}
Let $(\Omega,\sigalg{F},\filt{F},\P)$ be a filtered probability space, whose filtration $\filt{F}=(\sigalg{F}_t)_{t\geq 0}$ is assumed to be right-continuous, and let $\Sf$ denote the family of finite-valued stopping times on $(\Omega,\sigalg{F},\filt{F},\P)$. Unless indicated otherwise, stochastic processes are defined on $(\Omega,\sigalg{F},\filt{F},\P)$, and are adapted to $\filt{F}$, as well as being real-valued and c\`adl\`ag. We denote the family of local martingales by $\Mloc$, while $\M$ denotes the family of uniformly integrable martingales. Our concept of a local martingale corresponds with that of \citet[Definitions~I.1.33 and I.1.45]{JS03}, which implies that $\E(|M_0|)<\infty$, for all $M\in\Mloc$.\footnote{Several authors, including \citet[Section~I.6]{Pro05} and \citet[Definition~IV.1.5]{RY99}, allow for the possibility that the initial component of a local martingale may be non-integrable. This additional generality would add a technical overhead to the results that follow, without offering any compensating advantages.} The strict inclusion $\M\subsetneq\Mloc$ gives rise to the following problem, which is the focus of our study.

\begin{problem}
\label{probsec1:MainProblem}
Given $M\in\Mloc$, formulate necessary and sufficient conditions for determining whether $M\in\M$.
\end{problem}

A promising approach to Problem~\ref{probsec1:MainProblem} focuses on the weak tails of the suprema of the moduli of local martingales, as well as the weak tails of their quadratic variations. Several previous studies have employed this approach to derive necessary and sufficient conditions under which a given local martingale is a uniformly integrable martingale. However, the results obtained so far are not completely general, since the process in question is always required to satisfy additional regularity and/or integrability conditions. We begin with a brief survey of the relevant literature.

The following argument, which traces its origin to \citet{Rao69}, represents the starting point. Consider a continuous martingale $M=(M_t)_{t\geq 0}$ satisfying $\sup_{t\geq 0}\E(|M_t|)<\infty$. In that case, Doob's martingale convergence theorem ensures that the almost sure limit $M_\infty\coloneqq M_{\infty-}$ exists and satisfies $\E(|M_\infty|)<\infty$. Let
\begin{equation}
\label{eqsec1:FirstExitTime}
\tau_\lambda\coloneqq\inf\{t\geq 0\,|\,|M_t|>\lambda\}
\end{equation}
denote the first exit-time from the compact interval $[-\lambda,\lambda]$, for all $\lambda\geq 0$. Since $M^{\tau_\lambda}\coloneqq M_{\tau_\lambda\wedge\,\cdot}$ is a bounded martingale, and hence also a uniformly integrable martingale, for all $\lambda\geq 0$, it follows that
\begin{equation*}
\E(M^{\tau_\lambda}_0)=\E(M^{\tau_\lambda}_\infty)
=\lambda\P\biggl(\sup_{t\geq 0}|M_t|>\lambda\biggr)+\E\Bigl(\indicator{\{\sup_{t\geq 0}|M_t|\leq\lambda\}}M_\infty\Bigr).
\end{equation*}
Finally, an application of the dominated convergence theorem yields
\begin{equation*}
\lim_{\lambda\uparrow\infty}\lambda\P\biggl(\sup_{t\geq 0}|M_t|>\lambda\biggr)
=\E(M_0)-\E(M_\infty),
\end{equation*}
whence $M\in\M$ if and only if $\lim_{\lambda\uparrow\infty}\lambda\P\bigl(\sup_{t\geq 0}|M_t|>\lambda\bigr)=0$. \citet{AGY80} derived this result by means of a similar argument. In addition, they demonstrated that $M\in\M$ if and only if $\lim_{\lambda\uparrow\infty}\lambda\P\bigl(\PredQVar{M}^{\sfrac{1}{2}}_\infty\geq\lambda\bigr)=0$, where $\PredQVar{M}_\infty\coloneqq\PredQVar{M}_{\infty-}$. \citet{Nov81} independently obtained the same characterisations of uniformly integrable martingales, in the context of first-passage problems.

\citet{ELY97,ELY99} and \citet{Tak99} extended the previous results, by obtaining weak tail characterisations of uniformly integrable martingales within the class of continuous local martingales, provided the processes in question meet certain integrability requirements. Further generalisations were obtained by \citet{Nov97}, and \citet{LN06}, while \citet{Kaj07,Kaj08,Kaj09} presented weak tail characterisations of uniformly integrable martingales within the class of locally square-integrable martingales. Once again, the processes under consideration must satisfy a variety of additional integrability conditions for the results to be applicable.

We contribute to the literature surveyed above by deriving a weak tail characterisation of uniformly integrable martingales that solves Problem~\ref{probsec1:MainProblem} in full generality, without restricting to local martingales of a certain type. In particular, since the suite of conditions we present is both necessary and sufficient for an arbitrary local martingale to be a uniformly integrable martingale, our result is not subject to any restrictive integrability and/or regularity constraints. In detail, we prove the following theorem.

\begin{theorem}
\label{thmsec1:UIMart}
Let $M\in\Mloc$. Then $M\in\M$ if and only if
\begin{align}
\varliminf_{\lambda\uparrow\infty}\lambda\P\biggl(\sup_{t\geq 0}|M_t|>\lambda\biggr)&=0;\tag{A}\label{eqsec1:CondA}\\
\lim_{\lambda\uparrow\infty}\E\bigl(\indicator{\{\tau_\lambda<\infty\}}|\Delta M_{\tau_\lambda}|\bigr)&=0;\qquad\text{and}\tag{B}\label{eqsec1:CondB}\\
\E\biggl(\varliminf_{t\uparrow\infty}|M_t|\biggr)&<\infty,\tag{C}\label{eqsec1:CondC}
\end{align}
where $\Delta M\coloneqq M-M_{-}$ is the jump process associated with $M$.
\end{theorem}

It turns out that Conditions~\eqref{eqsec1:CondA}--\eqref{eqsec1:CondC} hold, with ordinary limits instead of limits inferior, if $M\in\M$. We may thus replace the limits inferior with ordinary limits in Conditions~\eqref{eqsec1:CondA} and \eqref{eqsec1:CondC}, without affecting the validity of Theorem~\ref{thmsec1:UIMart}.

Condition~\eqref{eqsec1:CondA} falls within the tradition of weak-tail characterisations of uniformly integrable martingales surveyed above, while Condition~\eqref{eqsec1:CondC} also appears in \citet{ruf15}. There it is shown that a local martingale $M$ is a uniformly integrable martingale if and only if $\E(|M_\tau|)<\infty$ and $\E(M_\tau)=\E(M_0)$, for all $\tau\in\Sf$, and Condition~\eqref{eqsec1:CondC} holds. Condition~\eqref{eqsec1:CondB} does not appear to have been considered before in the context of Problem~\ref{probsec1:MainProblem}.

Theorem~\ref{thmsec1:UIMart} also provides necessary and sufficient criteria for determining whether a given local martingale is a martingale, or a strict local martingale. To appreciate this, recall that $M\in\Mloc$ is a martingale if and only if $M^t\coloneqq M_{t\wedge\,\cdot}\in\M$, for all $t\geq 0$. Applying this observation to Theorem~\ref{thmsec1:UIMart} yields the following corollary.

\begin{corollary}
\label{corsec1:Martingales}
Let $M\in\Mloc$. Then $M$ is a martingale if and only if
\begin{align}
\varliminf_{\lambda\uparrow\infty}\lambda\P\biggl(\sup_{s\in[0,t]}|M_s|>\lambda\biggr)&=0;\tag{$\text{A}^\prime$}\label{eqsec1:CondA'}\\
\lim_{\lambda\uparrow\infty}\E\bigl(\indicator{\{\tau_\lambda\leq t\}}|\Delta M_{\tau_\lambda}|\bigr)&=0;\qquad\text{and}\tag{$\text{B}^\prime$}\label{eqsec1:CondB'}\\
\E(|M_t|)&<\infty,\tag{$\text{C}^\prime$}\label{eqsec1:CondC'}
\end{align}
for all $t\geq 0$.
\end{corollary}

The remainder of the article is structured as follows. We prove Theorem~\ref{thmsec1:UIMart} in Section~\ref{sec2}, while Section~\ref{sec3} presents several revealing examples. In particular, we construct examples of local martingales that are not uniformly integrable martingales, due to the selective failure of precisely one of Conditions~\eqref{eqsec1:CondA}--\eqref{eqsec1:CondC}. Finally, Section~\ref{sec4} exploits Theorem~\ref{thmsec1:UIMart} to present a systematic procedure for constructing uniformly integrable martingales whose suprema are not integrable.
\section{The Proof of Theorem~\ref{thmsec1:UIMart}}
\label{sec2}
In the lead-up to the proof of Theorem~\ref{thmsec1:UIMart}, we first explore some of the consequences of Conditions~\eqref{eqsec1:CondA}--\eqref{eqsec1:CondC}. To begin with, recall that a continuous local martingale that is stopped when first it leaves a compact interval is a bounded local martingale, and hence also a uniformly integrable martingale. The following lemma generalises this result.

\begin{lemma}
\label{lemsec2:Lemma1}
Suppose $M\in\Mloc$ satisfies Condition~\eqref{eqsec1:CondB}. Then $M^{\tau_\lambda}\in\M$, for all $\lambda\geq 0$.
\end{lemma}
\begin{proof}
Condition~\eqref{eqsec1:CondB} guarantees the existence of a $\lambda_*\geq 0$, such that $\E\bigl(\indicator{\{\tau_\lambda<\infty\}}|\Delta M_{\tau_\lambda}|\bigr)<\infty$, for all $\lambda\geq\lambda_*$, from which it follows that
\begin{equation*}
\E\biggl(\sup_{t\geq 0}|M^{\tau_\lambda}_t|\biggr)
\leq\E(|M_0|)+\lambda+\E\bigl(\indicator{\{\tau_\lambda<\infty\}}|\Delta M_{\tau_\lambda}|\bigr)
<\infty.
\end{equation*}
Given $\lambda\geq\lambda_*$, an application of the dominated convergence theorem then yields
\begin{equation*}
\lim_{K\uparrow\infty}\sup_{\sigma\in\Sf}\E\biggl(\indicator{\left\{\left|M^{\tau_\lambda}_\sigma\right|\geq K\right\}}|M^{\tau_\lambda}_\sigma|\biggr)
\leq\lim_{K\uparrow\infty}\E\biggl(\indicator{\left\{\sup_{t\geq 0}\left|M^{\tau_\lambda}_t\right|\geq K\right\}}\sup_{t\geq 0}|M^{\tau_\lambda}_t|\biggr)=0,
\end{equation*}
since $|M^{\tau_\lambda}_\sigma|\leq\sup_{t\geq 0}|M^{\tau_\lambda}_t|$, for all $\sigma\in\Sf$. In other words, $M^{\tau_\lambda}$ is a local martingale that belongs to class~(D) \citep[see e.g.][Definition~I.1.46]{JS03}, and is thus a uniformly integrable martingale \citep[see e.g.][Proposition~I.1.47]{JS03}. On the other hand, if $\lambda\in[0,\lambda_*]$, then $\tau_\lambda\leq\tau_{\lambda_*}$, whence  $M^{\tau_\lambda}=M^{\tau_{\lambda_*}\wedge\tau_\lambda}\in\M$, since $M^{\tau_{\lambda_*}}\in\M$ and the family of uniformly integrable martingales is stable under stopping.
\end{proof}

Of course, a local martingale may satisfy the conclusion of Lemma~\ref{lemsec2:Lemma1} without necessarily satisfying Condition~\eqref{eqsec1:CondB}. To illustrate this point, consider a non-negative local martingale $M\in\Mloc$ that does not satisfy Condition~\eqref{eqsec1:CondB} (see Example~\ref{exsec3:CondBFails} below), and fix $\lambda\geq 0$. Since $M$ is then a non-negative supermartingale, the random variable $M_{\tau_\lambda}\geq 0$ is well-defined and satisfies $\E(M_{\tau_\lambda})\leq\E(M_0)<\infty$ \citep[see e.g.][Theorem~I.1.39]{JS03}. Moreover,
\begin{equation*}
|M^{\tau_\lambda}_t|=M^{\tau_\lambda}_t
=\indicator{\{\tau_\lambda\leq t\}}M_{\tau_\lambda}+\indicator{\{\tau_\lambda>t\}}M_t
\leq M_{\tau_\lambda}+\lambda,
\end{equation*}
for all $t\geq 0$. These two observations imply that
\begin{equation*}
\E\biggl(\sup_{t\geq 0}|M^{\tau_\lambda}_t|\biggr)\leq\E(M_{\tau_\lambda})+\lambda<\infty,
\end{equation*}
from which it follows that $M^{\tau_\lambda}\in\M$.

Next, we establish two useful facts about local martingales for which Conditions~\eqref{eqsec1:CondA}--\eqref{eqsec1:CondB} hold, one of which is that such processes possess real-valued almost-sure limits.

\begin{lemma}
\label{lemsec2:Lemma2}
Suppose $M\in\Mloc$ satisfies Conditions~\eqref{eqsec1:CondA}--\eqref{eqsec1:CondB}. Then
\begin{equation*}
\varliminf_{\lambda\uparrow\infty}\E\bigl(\indicator{\{\tau_\lambda<\infty\}}|M^{\tau_\lambda}_\infty|\bigr)=0.
\end{equation*}
Moreover, the almost sure limit $M_\infty\coloneqq M_{\infty-}$ exists and satisfies $M_\infty\in\R$.
\end{lemma}
\begin{proof}
Note that the almost sure limit $M^{\tau_\lambda}_\infty\coloneqq M^{\tau_\lambda}_{\infty-}$ exists and satisfies $M^{\tau_\lambda}_\infty\in\R$, for all $\lambda\geq 0$, as a result of Lemma~\ref{lemsec2:Lemma1}. Now observe that
\begin{equation*}
\begin{split}
\varliminf_{\lambda\uparrow\infty}\E\bigl(&\indicator{\{\tau_\lambda<\infty\}}|M^{\tau_\lambda}_\infty|\bigr)
\leq\varliminf_{\lambda\uparrow\infty}\E\bigl(\indicator{\{\tau_\lambda<\infty\}}(|M_0|+\lambda+|\Delta M_{\tau_\lambda}|)\bigr)\\
&\leq\lim_{\lambda\uparrow\infty}\E\bigl(\indicator{\{\tau_\lambda<\infty\}}|M_0|\bigr)+\varliminf_{\lambda\uparrow\infty}\lambda\P\biggl(\sup_{t\geq 0}|M_t|>\lambda\biggr)+\lim_{\lambda\uparrow\infty}\E\bigl(\indicator{\{\tau_\lambda<\infty\}}|\Delta M_{\tau_\lambda}|\bigr)
=0,
\end{split}
\end{equation*}
by virtue of the dominated convergence theorem, and a direct application of Conditions~\eqref{eqsec1:CondA}--\eqref{eqsec1:CondB}. Given $\lambda\geq 0$, it also follows that
\begin{equation*}
\indicator{\{\tau_\lambda=\infty\}}M_{\infty-}
=\indicator{\{\tau_\lambda=\infty\}}M^{\tau_\lambda}_{\infty-}
=\indicator{\{\tau_\lambda=\infty\}}M^{\tau_\lambda}_\infty\in\R,
\end{equation*}
whence $\{M_{\infty-}\in\R\}\supseteq\{\tau_\lambda=\infty\}$. Consequently,
\begin{equation*}
\P(M_{\infty-}\in\R)\geq\lim_{\lambda\uparrow\infty}\P(\tau_\lambda=\infty)=1,
\end{equation*}
since Condition~\eqref{eqsec1:CondA} implies that $\lim_{\lambda\uparrow\infty}\P(\tau_\lambda<\infty)=0$. That is to say, the almost sure limit $M_\infty\coloneqq M_{\infty-}$ exists and satisfies $M_\infty\in\R$.
\end{proof}

Finally, we establish a convergence result that is used in the proof of Theorem~\ref{thmsec1:UIMart} to show that Conditions~\eqref{eqsec1:CondA}--\eqref{eqsec1:CondC} are sufficient for a local martingale to be a uniformly integrable martingale.

\begin{lemma}
\label{eqsec1:Lemma3}
Suppose $M\in\Mloc$ satisfies Conditions~\eqref{eqsec1:CondA}--\eqref{eqsec1:CondC}. Then $M_\infty\coloneqq M_{\infty-}$ exists, and
\begin{equation*}
\varliminf_{\lambda\uparrow\infty}\E(|M^{\tau_\lambda}_\infty-M_\infty|)=0.
\end{equation*}
\end{lemma}
\begin{proof}
An application of the dominated convergence theorem gives
\begin{equation*}
\E\Big(\lim_{\lambda\uparrow\infty}\indicator{\{\tau_\lambda<\infty\}}\Bigr)
=\lim_{\lambda\uparrow\infty}\P(\tau_\lambda<\infty)
=0,
\end{equation*}
by virtue of Condition~\eqref{eqsec1:CondA}, from which it follows that $\lim_{\lambda\uparrow\infty}\indicator{\{\tau_\lambda<\infty\}}=0$. Another application of the dominated convergence theorem then yields
\begin{equation*}
\lim_{\lambda\uparrow\infty}\E\bigl(\indicator{\{\tau_\lambda<\infty\}}|M_\infty|\bigr)=0,
\end{equation*}
since Lemma~\ref{lemsec2:Lemma2} and Condition~\eqref{eqsec1:CondC} ensure that $M_\infty\coloneqq M_{\infty-}$ exists and satisfies $\E(|M_\infty|)<\infty$. Finally, we observe that
\begin{equation*}
\begin{split}
\varliminf_{\lambda\uparrow\infty}\E(|M^{\tau_\lambda}_\infty-M_\infty|)
&=\varliminf_{\lambda\uparrow\infty}\E\bigl(\indicator{\{\tau_\lambda<\infty\}}|M^{\tau_\lambda}_\infty-M_\infty|\bigr)\\
&\leq\varliminf_{\lambda\uparrow\infty}\E\bigl(\indicator{\{\tau_\lambda<\infty\}}|M^{\tau_\lambda}_\infty|\bigr)+\lim_{\lambda\uparrow\infty}\E\bigl(\indicator{\{\tau_\lambda<\infty\}}|M_\infty|\bigr)
=0,
\end{split}
\end{equation*}
by virtue of Lemma~\ref{lemsec2:Lemma2} and the previous argument.
\end{proof}

The proof of Theorem~\ref{thmsec1:UIMart} verifies Conditions~\eqref{eqsec1:CondA}--\eqref{eqsec1:CondC} directly, for any uniformly integrable martingale, before using Lemma~\ref{eqsec1:Lemma3} to demonstrate that a local martingale satisfying those conditions is a uniformly integrable martingale.

\begin{proof}[Proof of Theorem~\ref{thmsec1:UIMart}]
$(\Rightarrow)$~Suppose $M\in\M$, in which case Condition~\eqref{eqsec1:CondC} holds immediately, since $M_\infty\coloneqq M_{\infty-}$ exists and satisfies $\E(|M_\infty|)<\infty$. Moreover, $|M|$ is a uniformly integrable submartingale, which implies that
\begin{equation}
\label{eqthmsec1:UIMart1}
\begin{split}
\E(|M_\infty|)\geq\E(|M_{\tau_\lambda}|)
&=\E\bigl(\indicator{\{\tau_\lambda<\infty\}}|M_{\tau_\lambda}|\bigr)+\E\bigl(\indicator{\{\tau_\lambda=\infty\}}|M_\infty|\bigr)\\
&\geq\lambda\P(\tau_\lambda<\infty)+\E\bigl(\indicator{\{\tau_\lambda=\infty\}}|M_\infty|\bigr),
\end{split}
\end{equation}
for all $\lambda\geq 0$. Next, by applying the monotone convergence theorem, followed by Doob's maximal inequalities, we obtain
\begin{equation}
\label{eqthmsec1:UIMart2}
\E\Bigl(\lim_{\lambda\uparrow\infty}\indicator{\{\tau_\lambda=\infty\}}\Bigr)
=\lim_{\lambda\uparrow\infty}\P(\tau_\lambda=\infty)
=1-\lim_{\lambda\uparrow\infty}\P\biggl(\sup_{t\geq 0}|M_t|>\lambda\biggr)\\
\geq 1-\lim_{\lambda\uparrow\infty}\frac{\E(|M_\infty|)}{\lambda}
=1,
\end{equation}
from which $\lim_{\lambda\uparrow\infty}\indicator{\{\tau_\lambda=\infty\}}=1$ follows. Combining this with \eqref{eqthmsec1:UIMart1} gives
\begin{equation*}
\lim_{\lambda\uparrow\infty}\lambda\P\biggl(\sup_{t\geq 0}|M_t|>\lambda\biggr)
=\lim_{\lambda\uparrow\infty}\lambda\P(\tau_\lambda<\infty)
\leq\E(|M_\infty|)-\lim_{\lambda\uparrow\infty}\E\bigl(\indicator{\{\tau_\lambda=\infty\}}|M_\infty|\bigr)
=0,
\end{equation*}
by an application of the monotone convergence theorem. In other words, Condition~\eqref{eqsec1:CondA} holds. Finally, the inequality $|\Delta M_{\tau_\lambda}|\leq 2|M_{\tau_\lambda}|$, for all $\lambda\geq 0$, together with the fact that $|M|$ is a uniformly integrable submartingale, yield
\begin{equation*}
\begin{split}
\lim_{\lambda\uparrow\infty}\E\bigl(\indicator{\{\tau_\lambda<\infty\}}|\Delta M_{\tau_\lambda}|\bigr)
\leq 2\lim_{\lambda\uparrow\infty}\E\bigl(\indicator{\{\tau_\lambda<\infty\}}|M_{\tau_\lambda}|\bigr)
&\leq2\lim_{\lambda\uparrow\infty}\E\bigl(\indicator{\{\tau_\lambda<\infty\}}|M_\infty|\bigr)\\
&=2\E\Bigl(\lim_{\lambda\uparrow\infty}\indicator{\{\tau_\lambda<\infty\}}|M_\infty|\Bigr)
=0,
\end{split}
\end{equation*}
by virtue of the dominated convergence theorem, since \eqref{eqthmsec1:UIMart2} implies that $\lim_{\lambda\uparrow\infty}\indicator{\{\tau_\lambda<\infty\}}=0$, and $\E(|M_\infty|)<\infty$. That is to say, Condition~\eqref{eqsec1:CondB} holds.
\vspace{2mm}\newline\noindent
$(\Leftarrow)$~Suppose $M\in\Mloc$ satisfies Conditions~\eqref{eqsec1:CondA}--\eqref{eqsec1:CondC}. Then
\begin{equation*}
0\leq\E\biggl(\varliminf_{\lambda\uparrow\infty}\E(|M^{\tau_\lambda}_\infty-M_\infty|\,|\,\sigalg{F}_t)\biggr)
\leq\varliminf_{\lambda\uparrow\infty}\E(|M^{\tau_\lambda}_\infty-M_\infty|)=0,
\end{equation*}
for all $t\geq 0$, by virtue of Fatou's lemma and Lemma~\ref{eqsec1:Lemma3}, from which it follows that
\begin{equation*}
\varliminf_{\lambda\uparrow\infty}\E(|M^{\tau_\lambda}_\infty-M_\infty|\,|\,\sigalg{F}_t)=0.
\end{equation*}
Consequently, there exists a non-decreasing sequence $(\lambda_n)_{n\in\N}$ of non-negative $\sigalg{F}_t$-measurable random variables, such that $\lim_{n\uparrow\infty}\lambda_n=\infty$ and
\begin{equation*}
\lim_{n\uparrow\infty}\E(|M^{\tau_{\lambda_n}}_\infty-M_\infty|\,|\,\sigalg{F}_t)=0,
\end{equation*}
for all $t\geq 0$. Hence,
\begin{equation*}
\E(M_\infty\,|\,\sigalg{F}_t)
=\lim_{n\uparrow\infty}\E(M^{\tau_{\lambda_n}}_\infty\,|\,\sigalg{F}_t)
=\lim_{n\uparrow\infty}M^{\tau_{\lambda_n}}_t=M_t,
\end{equation*}
for all $t\geq 0$, since Lemma~\ref{lemsec2:Lemma1} ensures that $M^{\tau_{\lambda_n}}\in\M$, for each $n\in\N$, and $\lim_{n\uparrow\infty}\tau_{\lambda_n}=\infty$, by \eqref{eqthmsec1:UIMart2}, since $\{\lim_{n\uparrow\infty}\tau_{\lambda_n}=\infty\}\supseteq\bigl\{\lim_{n\uparrow\infty}
\indicator{\{\tau_{\lambda_n}=\infty\}}=1\bigr\}$. This establishes that $M\in\M$.
\end{proof}

\section{Some Examples}
\label{sec3}
We begin by constructing three examples of local martingales for which precisely one of Conditions~\eqref{eqsec1:CondA}--\eqref{eqsec1:CondC} fails (a different one in each case), while the other two hold. In each case, Theorem~\ref{thmsec1:UIMart} legislates that the process in question is not a uniformly integrable martingale. This establishes the minimality of Conditions~\eqref{eqsec1:CondA}--\eqref{eqsec1:CondC}.

A Brownian motion is an obvious example of a local martingale for which Condition~\eqref{eqsec1:CondA} fails, while Conditions~\eqref{eqsec1:CondB}--\eqref{eqsec1:CondC} hold. However, the class of non-negative time-homogeneous regular diffusions in natural scale provides a more interesting family of examples. All processes of this type are local martingales that satisfy Conditions~\eqref{eqsec1:CondB}--\eqref{eqsec1:CondC}. However, since the limit in Condition~\eqref{eqsec1:CondA} is non-zero for such processes, it follows that they cannot be uniformly integrable martingales.

\begin{example}[Condition~\eqref{eqsec1:CondA} fails]
\label{exsec3:CondAFails1}
Let $X=(X_t)_{t\geq 0}$ be a non-negative time-homogeneous regular scalar diffusion in natural scale, with state-space $[0,\infty)$ or $(0,\infty)$, depending on its behaviour at the origin. Since such a process is continuous, it trivially satisfies Condition~\eqref{eqsec1:CondB}. Being in natural scale means that the scale function for $X$ is determined by $\scale(x)\coloneqq x$, for all $x>0$. This ensures that $X$ is a non-negative $\P_x$--local martingale, for all $x>0$, and consequently also a non-negative $\P_x$--supermartingale. As a result, it satisfies Condition~\eqref{eqsec1:CondC}. The fact that $X$ is a non-negative supermartingale imposes constraints on its behaviour at the origin. In particular, the origin is either an absorbing boundary or a natural boundary. In the former case the state space of $X$ is $[0,\infty)$, while it is $(0,\infty)$ in the latter case. Either way, we observe that
\begin{equation*}
\P_x\biggl(\sup_{t\geq 0}X_t>\lambda\biggr)=\P_x(\tau_\lambda<\infty)
=\lim_{l\downarrow 0}\P_x(\tau_\lambda<\tau_l)
=\lim_{l\downarrow 0}\frac{\scale(x)-\scale(l)}{\scale(\lambda)-\scale(l)}
=\frac{x}{\lambda},
\end{equation*}
for all $x>0$ and all $\lambda\geq x$, where $\P_x$ is the probability measure under which $X_0=x$.\footnote{There is a slight abuse of notation here, in the sense that $\tau_\lambda$ should be interpreted as the first-exit time \eqref{eqsec1:FirstExitTime} with $M$ replaced by $X$, for any $\lambda\geq 0$.} Consequently, we obtain
\begin{equation*}
\lim_{\lambda\uparrow\infty}\lambda\P_x\biggl(\sup_{t\geq 0}X_t>\lambda\biggr)=x>0,
\end{equation*}
for all $x>0$. That is to say, $X$ is not a uniformly integrable martingale, due to the failure of Condition~\eqref{eqsec1:CondA}.
\qed
\end{example}

The next example constructs a non-negative pure-jump martingale that is not a uniformly integrable martingale, since it satisfies Conditions~\eqref{eqsec1:CondA} and \eqref{eqsec1:CondC}, but not Condition~\eqref{eqsec1:CondB}. Starting with an initial value of one, the process jumps only at integer-valued times, while remaining constant over the intervening intervals. Negative jumps take it to zero, where it is absorbed, while the sizes of successive positive jumps grow combinatorially. To ensure that the resulting process is a martingale, the probabilities of positive jumps decrease very quickly.  

\begin{example}[Condition~\eqref{eqsec1:CondB} fails]
\label{exsec3:CondBFails}
Suppose $(\Omega,\sigalg{F},\P)$ supports a sequence $(Y_n)_{n\in\Z_+}$ of positive random variables, with $Y_0=1$ and 
\begin{equation}
\label{eqexsec3:CondBFails1}
\P(Y_n\in\d y)\coloneqq\frac{(n+1)!}{n}\indicator{(n!,(n+1)!]}(y)\frac{1}{y^2}\,\d y,
\end{equation}
for all $y\in\R_+$ and each $n\in\N$, as well as a sequence $(\xi_n)_{n\in\Z_+}$ of Bernoulli random variables, with $\xi_0=1$ and
\begin{equation}
\label{eqexsec3:CondBFails2}
%\P(\xi_n=1\,|\,\xi_0,\cdots,\xi_{n-1},Y_0,\cdots,Y_{n-1})\coloneqq\prod_{i=0}^{n-1}\xi_i\frac{Y_{n-1}}{\E(Y_n)},
\P(\xi_n=1\,|\,\xi_0,\cdots,\xi_{n-1},Y_0,\cdots,Y_{n-1})\coloneqq\frac{Y_{n-1}}{\E(Y_n)} \prod_{i=0}^{n-1}\xi_i
\end{equation}
for each $n\in\N$. We assume that $Y_n$ is independent of $\xi_0,\cdots,\xi_n$ and $Y_0,\cdots,Y_{n-1}$, for each $n\in\N$. The filtration $\filt{F}=(\sigalg{F}_t)_{t\geq 0}$ is determined by $\sigalg{F}_t\coloneqq\sigma(\xi_n,Y_n\,|\,0\leq n\leq\floor{t})$, for all $t\geq 0$, while the process $M=(M_t)_{t\geq 0}$ is specified by
\begin{equation*}
%M_t\coloneqq\prod_{i=0}^{\floor{t}}\xi_iY_{\floor{t}},	
M_t\coloneqq Y_{\floor{t}} \prod_{i=0}^{\floor{t}}\xi_i,	
\end{equation*}
for all $t\geq 0$. It follows that $M$ is adapted to $\filt{F}$, while the boundedness of $Y_n$, for each $n\in\Z_+$, ensures that $\E(|M_t|)<\infty$, for all $t\geq 0$. Also note that \eqref{eqexsec3:CondBFails2} implies that $\prod_{i=0}^n\xi_i=\xi_n$, for each $n\in\Z_+$, so that we may write $M_t=\xi_{\floor{t}}Y_{\floor{t}}$,	for all $t\geq 0$. This yields the useful identities
\begin{equation}
\label{eqexsec3:CondBFails3}
\indicator{\{M_n>0\}}=\indicator{\{\xi_n=1\}}=\xi_n,	
\end{equation}
for each $n\in\Z_+$. It also allows us to rewrite \eqref{eqexsec3:CondBFails2} as follows:
\begin{equation}
\label{eqexsec3:CondBFails4}
\P(\xi_n=1\,|\,\sigalg{F}_{n-1})=\frac{M_{n-1}}{\E(Y_n)},	
\end{equation}
for each $n\in\N$. It is now easy to see that $M$ is a martingale, since
\begin{align*}
\E(M_n\,|\,\sigalg{F}_{n-1})=\E(\xi_nY_n\,|\,\sigalg{F}_{n-1})
&=\E\big(\xi_n\E\bigl(Y_n\,\bigl|\,\sigma(\xi_n)\vee\sigalg{F}_{n-1}\bigr)\,\bigl|\,\sigalg{F}_{n-1}\bigr)\\
&=\E(\xi_n\,|\,\sigalg{F}_{n-1})\E(Y_n)
=\P(\xi_n=1\,|\,\sigalg{F}_{n-1})\E(Y_n)
=M_{n-1},
\end{align*}
for each $n\in\N$, by virtue of \eqref{eqexsec3:CondBFails3}, \eqref{eqexsec3:CondBFails4}, and the fact that $Y_n$ is independent of $\sigma(\xi_n)\vee\sigalg{F}_{n-1}$. Moreover, since $M$ is non-negative, Condition~\eqref{eqsec1:CondC} holds a fortiori. Next, we compute the probability that $M$ is strictly positive at any integer-valued time as follows:
\begin{equation*}
\P(M_n>0)=\P(\xi_n=1)=\E\bigl(\P(\xi_n=1\,|\,\sigalg{F}_{n-1})\bigr)
=\E\biggl(\frac{M_{n-1}}{\E(Y_n)}\biggr)=\frac{1}{\E(Y_n)},	
\end{equation*}
for each $n\in\N$, with the help of \eqref{eqexsec3:CondBFails3}, \eqref{eqexsec3:CondBFails4}, and the fact that $M$ is a martingale with $M_0=1$. Consequently, given $n\in\N$, we obtain 
\begin{equation*}
\P(M_n>\lambda)=\P(\xi_nY_n>\lambda)=\P(\xi_n=1,Y_n>\lambda)
=\P(\xi_n=1)\P(Y_n>\lambda)
=\frac{\P(Y_n>\lambda)}{\E(Y_n)},	
\end{equation*}
for all $\lambda\geq 0$, since $Y_n$ is independent of $\xi_n$. Now, given $\lambda>1$, let $n\in\N$ be the unique positive integer such that $n!<\lambda\leq(n+1)!$. In that case, the previous two identities, together with \eqref{eqexsec3:CondBFails1}, give
\begin{equation*}
\begin{split}
&\lambda\P\biggl(\sup_{t\geq 0}|M_t|>\lambda\biggr)
\leq\lambda\bigl(\P(M_n>\lambda)+\P(M_{n+1}>0)\bigr)
=\lambda\left(\frac{\P(Y_n>\lambda)}{\E(Y_n)}+\frac{1}{\E(Y_{n+1})}\right)\\
&\leq\frac{\lambda\P(Y_n>\lambda)}{\E(Y_n)}+\frac{(n+1)!}{\E(Y_{n+1})}\\
&=\lambda\biggl(\frac{(n+1)!}{n}\int_{\lambda}^{(n+1)!}\frac{1}{y^2}\,\d y\biggr)\biggl(\frac{(n+1)!}{n}\int_{n!}^{(n+1)!}\frac{1}{y}\,\d y\biggr)^{-1}+(n+1)!\biggl(\frac{(n+2)!}{n+1}\int_{(n+1)!}^{(n+2)!}\frac{1}{y}\,\d y\biggr)^{-1}\\
&\leq\lambda\biggl(\frac{(n+1)!}{n}\frac{1}{\lambda}\biggr)\biggl(\frac{(n+1)!}{n}\ln(n+1)\biggr)^{-1}+(n+1)!\biggl(\frac{(n+2)!}{n+1}\ln(n+2)\biggr)^{-1}\\
&=\frac{1}{\ln(n+1)}+\frac{n+1}{(n+2)\ln(n+2)}
<\frac{2}{\ln(n+1)},
\end{split}  
\end{equation*}
by virtue of the inclusion $\{\sup_{t\geq 0}M_t>\lambda\}\subseteq\{M_n>\lambda\}\cup\{M_{n+1}>0\}$. Consequently,
\begin{equation*}
\lim_{\lambda\uparrow\infty}\lambda\P\biggl(\sup_{t\geq 0}|M_t|>\lambda\biggr)
\leq\lim_{n\uparrow\infty}\frac{2}{\ln(n+1)}=0,
\end{equation*}
which establishes that $M$ satisfies Condition~\eqref{eqsec1:CondA}. Finally, given $n\in\N$, we use the identities $\xi_{n+1}^2=\xi_{n+1}$ and $\xi_{n+1}\xi_n=\xi_{n+1}\prod_{i=0}^n\xi_i=\prod_{i=0}^{n+1}\xi_i=\xi_{n+1}$ to get
\begin{equation*}
\begin{split}
\E\bigl(&\indicator{\{\tau_{n!}<\infty\}}|\Delta M_{\tau_{n!}}|\bigr)
=\E\bigl(\indicator{\{M_n>0\}}\Delta M_n\bigr)
=\E(\xi_n\Delta M_n)
=\E\bigl(\xi_n(\xi_nY_n-\xi_{n-1}Y_{n-1})\bigr)\\
&=\E\bigl(\xi_n(Y_n-Y_{n-1})\bigr)
=\E(M_n)-\E\bigl(\P(\xi_n=1\,|\,\sigalg{F}_{n-1})Y_{n-1}\bigr)
=1-\E\biggl(\frac{M_{n-1}}{\E(Y_n)}Y_{n-1}\biggr)\\
&\geq 1-\E\biggl(\frac{M_{n-1}}{\E(Y_n)}(n-1)!\biggr)
=1-\frac{(n-1)!}{\E(Y_n)}
=1-(n-1)!\biggl(\frac{(n+1)!}{n}\int_{n!}^{(n+1)!}\frac{1}{y}\,\d y\biggr)^{-1}\\
&=1-(n-1)!\biggl(\frac{(n+1)!}{n}\ln(n+1)\biggr)^{-1}
=1-\frac{1}{(n+1)\ln(n+1)},
\end{split}	
\end{equation*}
with the help of \eqref{eqexsec3:CondBFails1}, \eqref{eqexsec3:CondBFails3} and \eqref{eqexsec3:CondBFails4}, and the fact that $M$ is a martingale. Hence,
\begin{equation*}
\varlimsup_{\lambda\uparrow\infty}\E\bigl(\indicator{\{\tau_\lambda<\infty\}}|\Delta M_{\tau_\lambda}|\bigr)
\geq1-\lim_{n\uparrow\infty}\frac{1}{(n+1)\ln(n+1)}=1,
\end{equation*}
from which we deduce that $M$ does not satisfy Condition~\eqref{eqsec1:CondB}. So $M$ is a non-negative martingale that satisfies Conditions~\eqref{eqsec1:CondA} and \eqref{eqsec1:CondC}, but not Condition~\eqref{eqsec1:CondB}, and is thus not a uniformly integrable martingale.
\qed
\end{example}

The next example presents a continuous local martingale that satisfies Conditions~\eqref{eqsec1:CondA} and \eqref{eqsec1:CondB}, but not Condition~\eqref{eqsec1:CondC}. It elaborates on an example originally due to  \citet{AGY80}.

\begin{example}[Condition~\eqref{eqsec1:CondC} fails]
\label{exsec3:CondCFails}
Let $B$ be a one-dimensional Brownian motion defined on $(\Omega,\sigalg{F},\filt{F},\P)$, and suppose the sigma-algebra $\sigalg{F}_0$ accommodates a discrete random variable $Y$, whose distribution is determined by
\begin{equation*}
\P(Y=n)\coloneqq\frac{c}{n^2\ln(n+2)},	
\end{equation*}
for each $n\in\N$, where
\begin{equation*}
c\coloneqq\biggl(\sum_{i=1}^\infty\frac{1}{i^2\ln(i+2)}\biggr)^{-1}.	
\end{equation*}
Now let
\begin{equation*}
\rho\coloneqq\inf\{t\geq 0\,|\,|B_t|=Y\}	
\end{equation*}
denote the first hitting time of $Y$ by $|B|$, and note that $\rho<\infty$. The definition of $Y$ ensures that
\begin{equation*}
\begin{split}
n\P(Y\geq n)=n\sum_{j=n}^\infty\frac{c}{j^2\ln(j+2)}
\leq\frac{cn}{\ln(n+2)}\sum_{j=n}^\infty\frac{1}{j^2}
\leq\frac{cn}{\ln(n+2)}\int_{n-1}^\infty\frac{1}{x^2}\,\d x
&=\frac{cn}{(n-1)\ln(n+2)}\\
&\leq\frac{2c}{\ln(n+2)},
\end{split}
\end{equation*}
for each $n\in\N$. The martingale $M\coloneqq B^\rho$ then satisfies Condition~\eqref{eqsec1:CondA}, since
\begin{equation*}
\lim_{\lambda\uparrow\infty}\lambda\P\biggl(\sup_{t\geq 0}|M_t|>\lambda\biggr)
=\lim_{\lambda\uparrow\infty}\lambda\P\biggl(\sup_{t\geq 0}|B^\rho_t|>\lambda\biggr)
=\lim_{\lambda\uparrow\infty}\lambda\P(|B_\rho|>\lambda)
=\lim_{n\uparrow\infty}n\P(Y\geq n)=0.
\end{equation*}
Moreover, $M$ satisfies Condition~\eqref{eqsec1:CondB}, by virtue of its continuity. Based on these observations, Lemma~\ref{lemsec2:Lemma2} ensures that $M_\infty\coloneqq M_{\infty-}$ exists and satisfies $M_\infty=B_\rho=\pm Y$. However,
\begin{equation*}
\E(|M_\infty|)=\E(Y)
=\sum_{n=1}^\infty\frac{c}{n\ln(n+2)}=\infty
\end{equation*}
implies that $M$ does not satisfy Condition~\eqref{eqsec1:CondC}, which implies that it cannot be a uniformly integrable martingale.
\qed
\end{example}

Consider a non-negative time-homogeneous regular scalar diffusion $X$ in natural scale, as in Example~\ref{exsec3:CondAFails1}. Although we have established that $X$ is not a uniformly integrable martingale, the question of whether it is a martingale or a strict local martingale is often important in applications. First note that $X$ satisfies Conditions~\eqref{eqsec1:CondB'} and \eqref{eqsec1:CondC'}, for the same reasons that it satisfies Conditions~\eqref{eqsec1:CondB} and \eqref{eqsec1:CondC}. According to Corollary~\ref{corsec1:Martingales}, it is thus a $\P_x$-martingale, for all $x>0$, if and only if it satisfies Condition~\eqref{eqsec1:CondA'}.

Fortunately, it is straightforward to check whether $X$ satisfies Condition~\eqref{eqsec1:CondA'} in this setting. First, given $\alpha>0$, we recall the identity
\begin{equation*}
\E_x(\e^{-\alpha\tau_\lambda})=\frac{\psi_\alpha(x)}{\psi_\alpha(\lambda)},
\end{equation*}
for all $x>0$ and all $\lambda\geq x$, where $\psi_\alpha$ is the unique (up to multiplication by a positive scalar) monotonically increasing solution to the ordinary differential equation
\begin{equation}
\label{eqsec3:ODE}
\frac{\d}{\d\speed}\frac{\d}{\d x}\psi_\alpha(x)=\alpha\psi_\alpha(x),
\end{equation}
for all $x>0$ \citep[see e.g.][Section~II.1]{BS02}.\footnote{Note that $\tau_\lambda$ should once again be interpreted as the first-exit time \eqref{eqsec1:FirstExitTime} with $M$ replaced by $X$, for any $\lambda\geq 0$.} Next, we observe that
\begin{equation*}
0\leq\lambda\P_x(\tau_{\lambda}<t)\leq\lambda\P_x(\tau_{\lambda}<\infty)	
=\lambda\P_x\biggl(\sup_{t\geq 0}X_t>\lambda\biggr)\leq x,
\end{equation*}
for all $x>0$ and all $\lambda\geq 0$, by virtue of the maximal inequality for non-negative supermartingales \citep[see e.g.][Exercise~II.1.15]{RY99}. Moreover, the Laplace transform of the upper bound above is finite:
\begin{equation*}
\Laplace{\alpha}\{x\}=\int_0^\infty\e^{-\alpha t}x\,\d t=\frac{x}{\alpha}<\infty,
\end{equation*}
for all $x>0$. We may thus use the dominated convergence theorem to get
\begin{equation*}
\begin{split}
\Laplace{\alpha}\biggl\{&\lim_{\lambda\uparrow\infty}\lambda\P_x\biggl(\sup_{s\in[0,t]}X_s>\lambda\biggr)\biggr\}
=\Laplace{\alpha}\Bigl\{\lim_{\lambda\uparrow\infty}\lambda\P_x(\tau_\lambda<t)\Bigr\}
=\lim_{\lambda\uparrow\infty}\lambda\Laplace{\alpha}\{\P_x(\tau_\lambda<t)\}	\\
&=\lim_{\lambda\uparrow\infty}\lambda\int_0^\infty\e^{-\alpha t}\P_x(\tau_\lambda<t)\,\d t
=\lim_{\lambda\uparrow\infty}\frac{\lambda}{\alpha}\int_0^\infty\e^{-\alpha t}\P_x(\tau_\lambda\in\d t)
=\lim_{\lambda\uparrow\infty}\frac{\lambda}{\alpha}\E_x(\e^{-\alpha\tau_\lambda})\\
&=\lim_{\lambda\uparrow\infty}\frac{\lambda}{\alpha}\frac{\psi_\alpha(x)}{\psi_\alpha(\lambda)}
=\frac{1}{\alpha}\frac{\psi_\alpha(x)}{\psi_\alpha'(\infty-)},
\end{split}
\end{equation*}
for all $x>0$, where the fourth equality follows from the integration by parts formula, and the final equality is an application of L'H\^opital's rule. Finally, the uniqueness of Laplace transforms ensures that Condition~\eqref{eqsec1:CondA'} holds if and only if $\psi_\alpha'(\infty-)=\infty$.

\citet{DS02a} obtained a popular characterisation of martingales within the class of driftless non-negative time-homogeneous It\^o diffusions. \citet{Kot06} subsequently extended their result to cover the class of all non-negative time-homogeneous regular scalar diffusions $X$ in natural scale, by showing that $X$ is a $\P_x$ martingale, for all $x>0$, if and only if
\begin{equation*}
\int_1^\infty x\,\speed(\d x)=\infty.	
\end{equation*}
It follows that this condition is equivalent to $\psi_\alpha'(\infty-)=\infty$ \citep[see][for a formal demonstration of the equivalence]{HP11}. We employ these two criteria below to investigate the martingale properties of two specific non-negative time-homogeneous regular scalar diffusions in natural scale. First we demonstrate that a driftless geometric Brownian motion is a martingale.

\begin{example}[Condition~\eqref{eqsec1:CondA'} holds]
\label{exsec3:CondsA'Holds}
Let $X$ be a driftless geometric Brownian motion. That is to say, the state-space of $X$ is $(0,\infty)$, and its scale function and speed measure are given by
\begin{equation*}
\scale(x)\coloneqq x\qquad\text{and}\qquad\speed(\d x)\coloneqq\frac{2}{x^2}\,\d x,
\end{equation*}
for all $x>0$. Furthermore, given any $\alpha>0$, the monotonically increasing solution to \eqref{eqsec3:ODE} is
\begin{equation*}
\psi_\alpha(x)\coloneqq x^{\frac{1}{2}(\sqrt{8\alpha+1}+1)},
\end{equation*}
for all $x>0$. Then
\begin{equation*}
\int_1^\infty x\,\speed(\d x)=2\ln(\infty-)=\infty
\end{equation*}
implies that $X$ is a $\P_x$-martingale, for all $x>0$. We arrive at the same conclusion by observing that
\begin{equation*}
\psi_\alpha'(\infty-)=\lim_{x\uparrow\infty}\frac{1}{2}\bigl(\sqrt{8\alpha+1}+1\bigr)x^{\frac{1}{2}(\sqrt{8\alpha+1}-1)}=\infty,
\end{equation*}
or by noting that $\infty$ is a natural boundary for $X$.
\qed
\end{example}

Our second example is the inverse of a Bessel process of dimension three. We demonstrate formally that such a process is a strict local martingale:

\begin{example}[Condition~\eqref{eqsec1:CondA'} fails]
\label{exsec3:CondsA'Fails}
Let $X$ be the inverse of a Bessel process of dimension three. Its state-space is $(0,\infty)$, while its scale function and speed measure are given by
\begin{equation*}
\scale(x)\coloneqq x\qquad\text{and}\qquad\speed(\d x)\coloneqq\frac{2}{x^4}\,\d x,
\end{equation*}
for all $x>0$. In that case, the monotonically increasing solution to \eqref{eqsec3:ODE} is given by
\begin{equation*}
\psi_\alpha(x)\coloneqq x^{-\frac{\sqrt{2\alpha}}{x}},
\end{equation*}
for all $x>0$, where $\alpha>0$ is given. Since 
\begin{equation*}
\int_1^\infty x\,\speed(\d x)=1-\lim_{x\uparrow\infty}\frac{1}{x^2}=1<\infty,
\end{equation*}
it follows that $X$ is a strict $\P_x$-local martingale, for all $x>0$. This is confirmed by observing that 
\begin{equation*}
\psi_\alpha'(\infty-)=\lim_{x\uparrow\infty}\biggl(1+\frac{\sqrt{2\alpha}}{x}\biggr)x^{-\frac{\sqrt{2\alpha}}{x}}=1<\infty,
\end{equation*}
or by noting that $\infty$ is an entrance boundary for $X$.
\qed
\end{example}

\section{A Remark on $\Hp{1}$ Martingales}
\label{sec4}
Following \citet[][Section~VII.3]{DM82}, let $\Hp{1}$ denote the family of local martingales $M\in\Mloc$, for which $\E\bigl(\sup_{t\geq 0}|M_t|\bigr)<\infty$.\footnote{\citet[][Section~VII.3]{DM82} actually define $\Hp{1}$ to be the family of martingales satisfying $\E\bigl(\sup_{t\geq 0}|M_t|\bigr)<\infty$. However, it is easy to see that every local martingale satisfying that condition is in fact a (uniformly integrable) martingale \citep[see e.g.][Theorem~I.51]{Pro05}.} Since
\begin{equation*}
\E\biggl(\sup_{t\geq 0}|M_t|\biggr)=\int_0^\infty\P\biggl(\sup_{t\geq 0}|M_t|>t\biggr)\,\d t
\end{equation*}
and
\begin{equation*}
\sum_{n=1}^\infty\P\biggl(\sup_{t\geq 0}|M_t|>n\biggr)
\leq\int_0^\infty\P\biggl(\sup_{t\geq 0}|M_t|>t\biggr)\,\d t
\leq 1+\sum_{n=1}^\infty\P\biggl(\sup_{t\geq 0}|M_t|>n\biggr),
\end{equation*}
it follows that a local martingale $M\in\Mloc$ belongs to $\Hp{1}$ if and only if it satisfies the condition
\begin{equation}
\label{eqsec3:CondD}
\sum_{n=1}^\infty\P\biggl(\sup_{t\geq 0}|M_t|>n\biggr)<\infty.\tag{D}\end{equation}
While it is clear that $\Hp{1}\subseteq\M$ \citep[see e.g][Theorem~I.51]{Pro05}, the reverse inclusion does not hold. Interpreted in the light of Theorem~\ref{thmsec1:UIMart}, this implies that Conditions~\eqref{eqsec1:CondA}--\eqref{eqsec1:CondC} must hold for every local martingale that satisfies Condition~\eqref{eqsec3:CondD}, but not conversely. The following example exploits this observation to provide a recipe for constructing non-negative uniformly integrable martingales in $\M\setminus\mathscr{H}^1$.

\begin{example}[Condition~\eqref{eqsec3:CondD} fails]
\label{exsec4:CondDFails}
Fix a non-negative local martingale $M\in\Mloc\setminus\M$ that is not a uniformly integrable martingale, and define the non-decreasing sequence $(c_n)_{n\in\N}\subset(1,\infty)$ by setting
\begin{equation}
\label{eqexsec4:CondDFails1}
c_n\coloneqq \ln\biggl(\e+\sum_{k=1}^n\P\biggl(\sup_{t\geq 0}|M_t|>k\biggr)\biggr),	
\end{equation}
for each $n\in\N$. Since $M\notin\Hp{1}$, it follows that $\lim_{n\uparrow\infty}c_n=\infty$. Next, suppose that $(\Omega,\sigalg{F},\P)$ accommodates a discrete $\sigalg{F}_0$-measurable random variable $Y\in\N$ that is independent of $M$, and whose distribution satisfies $\P(Y>n)=\sfrac{1}{c_n}$, for each $n\in\N$, and let
\begin{equation}
\label{eqexsec4:CondDFails2}
\sigma\coloneqq\inf\{t\geq 0\,|\,|M_t|>Y\}	
\end{equation}
denote the first time $M$ exceeds $Y$. It follows that 
\begin{equation*}
\P\biggl(\sup_{t\geq 0}|M^\sigma_t|>n\biggr)
\geq\P\biggl(\sup_{t\geq 0}|M_t|>n\biggr)\P(Y>n)
=\frac{1}{c_n}\P\biggl(\sup_{t\geq 0}|M_t|>n\biggr),
\end{equation*}
for each $n\in\N$. Consequently,
\begin{equation*}
\sum_{n=1}^\infty\P\biggl(\sup_{t\geq 0}|M^\sigma_t|>n\biggr)
\geq\lim_{m\uparrow\infty}\frac{1}{c_m}\sum_{n=1}^m\P\biggl(\sup_{t\geq 0}|M_t|>n\biggr)
=\lim_{m\uparrow\infty}\frac{\e^{c_m}-\e}{c_m}=\infty,
\end{equation*}
since $(c_n)_{n\in\N}$ is non-decreasing and $\lim_{n\uparrow\infty}c_n=\infty$. This implies that $M^\sigma\notin\Hp{1}$. On the other hand, the almost sure limit $M^\sigma_\infty\coloneqq M^\sigma_{\infty-}\in\R$ exists, since $M^\sigma$ is a non-negative local martingale, and hence also a non-negative supermartingale. Moreover,
\begin{equation*}
\begin{split}
\E(M^\sigma_\infty)&=\sum_{n=1}^\infty\E(M^\sigma_\infty\,|\,Y=n)\P(Y=n)
=\sum_{n=1}^\infty\E(M^{\tau_n}_\infty\,|\,Y=n)\P(Y=n)\\
&=\sum_{n=1}^\infty\E(M^{\tau_n}_\infty)\P(Y=n)
=\sum_{n=1}^\infty\E(M_0)\P(Y=n)
=\E(M^\sigma_0),
\end{split}
\end{equation*}
where the penultimate equality follows from the fact that $M^{\tau_n}\in\M$, for each $n\in\N$, by virtue of the discussion following Lemma~\ref{lemsec2:Lemma1}. This establishes that $M^\sigma\in\M$.
\qed
\end{example}

Consider a local martingale $M\in\Mloc$ that belongs to Class~($\mathcal{C}_0$), according to the terminology of \citet{NY06}, and suppose that $M_0=1$. In that case, $M$ is a strictly positive local martingale without any positive jumps, for which $M_\infty\coloneqq M_{\infty-}=0$. The construction in Example~\ref{exsec4:CondDFails} is then applicable, since $\E(M_\infty)=0<1=\E(M_0)$ implies that $M\in\Mloc\setminus\M$. Moreover, an application of Doob's maximal identity \citep[see][Lemma~2.1]{NY06} provides the following concrete representation for the non-decreasing sequence $(c_n)_{n\in\N}$, defined by \eqref{eqexsec4:CondDFails1}:
\begin{equation*}
c_n=\ln\biggl(\e+\sum_{k=1}^n\frac{1}{k}\biggr),	
\end{equation*}
for each $n\in\N$. It is then straightforward to see that $\lim_{n\uparrow\infty}c_n=\infty$, which is the crucial ingredient for showing that $M^\sigma\notin\Hp{1}$, where the stopping time $\sigma$ is given by \eqref{eqexsec4:CondDFails2}.

% The references.
\bibliography{ProbFinBiblio}
\bibliographystyle{chicago}
\end{document}